\documentclass[12pt,a4paper]{article}
\usepackage{amsmath,amstext,amssymb,amscd, color}
\usepackage{graphicx}

\usepackage[russian]{babel}

\newcommand{\Xcomment}[1]{}

\newtheorem{theorem}{Theorem}[section]
\newtheorem{lemma}[theorem]{Lemma}
\newtheorem{corollary}[theorem]{Corollary}

\newtheorem{prop}[theorem]{Proposition}

\newcommand{\SEC}[1]{\ref{sec:#1}}  
\newcommand{\SSEC}[1]{\ref{ssec:#1}}  

\makeatletter \@addtoreset{equation}{section} \makeatother

\newenvironment{proof}{\noindent{\bf Proof}~}%
{\hfill$\qed$\medskip}

\def\qed{ \ \vrule width.1cm height.3cm depth0cm}

\newenvironment{numitem1}{\refstepcounter{equation}\begin{enumerate}%
\item[(\thesection.\arabic{equation})]}{\end{enumerate}}

\newcommand{\refeq}[1]{(\ref{eq:#1})}  

 \makeatletter
\renewcommand{\section}{\@startsection{section}{1}{0pt}%
{-3.5ex plus -1ex minus -.2ex}{2.3ex plus .2ex}%
{\normalfont\Large}}
 \makeatother

 \makeatletter
\renewcommand{\subsection}{\@startsection{subsection}{2}{0pt}%
{-3.0ex plus -1ex minus -.2ex}{-1.5ex plus .2ex}%
{\normalfont\normalsize\bf}}
 \makeatother

  \Xcomment{
 \makeatletter
\renewcommand{\subsection}{\@startsection{subsection}{2}{0pt}%
{-3.0ex plus -1ex minus -.2ex}{1.5ex plus .2ex}%
{\normalfont\normalsize\bf}}
 \makeatother
 }

 \makeatletter
\renewcommand{\subsubsection}{\@startsection{subsubsection}{2}{0pt}%
{-2.0ex plus -1ex minus -.2ex}{-2.0ex plus .2ex}%
{\normalfont\normalsize\underline}}
 \makeatother

\def\Rset{{\mathbb R}}
\def\Zset{{\mathbb Z}}

\def\Ascr{{\cal A}}
\def\Bscr{{\cal B}}
\def\Cscr{{\cal C}}

\def\Sscr{{\cal S}}

\def\Qbold{{\bf Q}}
\def\Sbold{{\bf S}}

\def\Inver{{\rm Inv}}

\def\bottom{{\rm bt}}
\def\Spec{{\rm Sp}}

\def\tilde{\widetilde}
\def\hat{\widehat}
\def\bar{\overline}

\def\Qst{Q^{\rm st}}
\def\Qant{Q^{\rm ant}}
\def\bd{\partial}

\def\Pack{{\rm Pac}}

\def\Zfr{Z^{\,\rm fr}}
\def\Zrear{Z^{\,\rm rear}}
\def\Zrim{Z^{\,\rm rim}}
\def\Cfr{C^{\,\rm fr}}
\def\Crear{C^{\,\rm rear}}

\begin{document}

 \title{Cubillages on cyclic zonotopes, membranes, \\
                and higher separation}

 \author{Vladimir I.~Danilov\thanks{Central Institute of Economics and
Mathematics of the RAS, 47, Nakhimovskii Prospect, 117418 Moscow, Russia;
email: danilov@cemi.rssi.ru.}
 \and
Alexander V.~Karzanov\thanks{Central Institute of Economics and Mathematics of
the RAS, 47, Nakhimovskii Prospect, 117418 Moscow, Russia; email:
akarzanov7@gmail.com. Corresponding author. }
  \and
Gleb A.~Koshevoy\thanks{The Institute for Information Transmission Problems of
the RAS, 19, Bol'shoi Karetnyi per., 127051 Moscow, Russia; email:
koshevoyga@gmail.com. Supported in part by grant RSF 16-11-10075. }
 }

\date{}

 \maketitle

 \begin{quote}
 {\bf Abstract.} \small
We study certain structural properties of fine zonotopal tilings, or
\emph{cubillages}, on cyclic zonotopes $Z(n,d)$ of an arbitrary dimension $d$
and their relations to $(d-1)$-separated collections of subsets of a set
$\{1,2,\ldots,n\}$. (Collections of this sort are well known as \emph{strongly
separated} ones when $d=2$, and as \emph{chord separated} ones when $d=3$.)
 \smallskip

{\em Keywords}\,: zonotope, cubillage, higher Bruhat order, strongly separated
sets, chord separated sets

\smallskip {\em MSC Subject Classification}\, 05E10, 05B45
 \end{quote}

\baselineskip=15pt
\parskip=2pt

\section{Introduction}  \label{sec:intr}

We consider a generalization of the notions of strongly separated and chord
separated set-systems. Let $n$ be a positive integer and denote the set
$\{1,2,\ldots,n\}$ by $[n]$.
  \Xcomment{
For a set $X\subseteq[n]$, let $\min(X)$ ($\max(X)$) denote the minimal (resp.
maximal) element of $X$, letting $\min(X)=\max(X):=0$ if $X=\emptyset$. For
subsets $A,B\subseteq[n]$, we will write $A<B$ if $\max(A)<\min(B)$.
  }
  \smallskip

\noindent\textbf{Definition.} ~Let $r\in[n-1]$. Two sets $A,B\subseteq [n]$ are
called $r$-\emph{separated} (from each other) if there is no sequence
$i_0<i_1<\cdots<i_{r+1}$ of elements of $[n]$ such that the elements with even
indices (namely, $i_0,i_2,\ldots$) and the elements with odd indices
($i_1,i_3,\ldots$) belong to different sets among $A-B$ and $B-A$ (where
$A'-B'$ denotes the set difference $\{i\colon A'\ni i\notin B'\}$). In other
words, one can choose $r'\le r$ integers (``separating points'') $a_1\le
a_2\le\cdots\le a_{r'}$ in $[n]$ such that the intervals $[a_i,a_{i+1}]$ with
$i$ even cover one of $A-B$ and $B-A$, while the ones with $i$ odd cover the
other of these sets, where $a_{r'+1}:=n$ and $[a,b]$ denotes
$\{a,a+1,\ldots,b\}$. Accordingly, a collection (set-system) $\Ascr\subseteq
2^{[n]}$ is called $r$-separated if any two of its members are such.
  \smallskip

We denote the set of all inclusion-wise maximal $r$-separated collections
$\Ascr$ in $2^{[n]}$ as $\Sbold_{n,r+1}$, and the maximal size $|\Ascr|$ of
such an $\Ascr$ by $s_{n,r+1}$ (for technical reasons, we prefer to use the
subscript pair $(n,r+1)$ rather than $(n,r)$). When all collections in
$\Sbold_{n,r+1}$ are of the same size, $\Sbold_{n,r+1}$ is said to be
\emph{pure}.

In particular, $\Sbold_{n,n}$ consists of the unique collection $2^{[n]}$
(since any two subsets of $[n]$ are $(n-1)$-separated), giving the simplest
purity case.

The concept of 1-separation was introduced, under the name of \emph{strong
separation}, by Leclerc and Zelevinsky~\cite{LZ} who proved the important fact
that
  \begin{numitem1} \label{eq:LZ}
for any $n\ge 2$, the set $\Sbold_{n,2}$ is pure (and $s_{n,2}$ equals
$\binom{n}{0}+ \binom{n}{1}+\binom{n}{2}=\frac12 n(n+1)+1$).
  \end{numitem1}

Recently an analogous purity result on 2-separation was shown by
Galashin~\cite{gal}:
  \begin{numitem1} \label{eq:gal}
for any $n\ge 3$, the set $\Sbold_{n,3}$ is pure (and $s_{n,3}=
\binom{n}{0}+\binom{n}{1}+ \binom{n}{2}+\binom{n}{3}$).
  \end{numitem1}

(In~\cite{gal}, 2-separated sets $A,B\subseteq[n]$ are called \emph{chord
separated}, which is justified by the observation that if $n$ points
$1,2,\ldots,n$ are disposed on a circumference $O$, in this order cyclically,
then there is a chord to $O$ separating $A-B$ from $B-A$.)

However, a nice purity behavior as above for $r$-separated set-systems with
$r=1,2$ is not extended in general to larger $r$'s, as it follows from profound
results on oriented matroids due to Galashin and Postnikov~\cite{GP}. Being
specified to $r$-separated set-systems, the following property is obtained.
  \begin{theorem} \label{tm:GP} {\rm~\cite{GP}}
~$\Sbold_{n,r+1}$ is pure if and only if $\min\{r,n-r\}\le 2$.
  \end{theorem}

One purpose of this paper is to give another proof of Theorem~\ref{tm:GP}
(which looks rather transparent modulo appealing to~\refeq{LZ}
and~\refeq{gal}).

In fact, the content of this paper is wider. In particular, we are going to
demonstrate representable cases of extendable and non-extendable set-systems.
Here we say that $\Ascr\subseteq 2^{[n]}$ is $(n,r+1)$-\emph{extendable} if
there exists  a maximal by size $r$-separated collection in $2^{[n]}$ including
$\Ascr$. (So $\Sbold_{n,r+1}$ is pure if and only if any $r$-separated
set-system in $2^{[n]}$ is $(n,r+1)$-extendable.)

Our study of separated set-systems is based on a geometric approach whose
theoretical grounds were originated in the classical work by Manin and
Schechtman~\cite{MS} where \emph{higher Bruhat orders}, generalizing weak
Bruhat ones, were introduced and well studied. (Recall that the higher Bruhat
order for $(n,d)$ compares certain (so-called ``packet admissible'') total
orders on the set $\binom{[n]}{d}$ of $d$-elements subsets of $[n]$, and it
turns into the weak one when $d=1$, which compares permutations on $[n]$.)
Subsequently Voevodskij and Kapranov~\cite{VK} and Ziegler~\cite{zieg} gave
nice geometric interpretations and established additional important results.

Based on these sources, we deal with a \emph{cyclic zonotope} $Z=Z(n,d)$, that
is the Minkowski sum of $n$ segments in $\Rset^d$ forming a \emph{cyclic
configuration}, and consider a fine zonotopal tiling, that is a subdivision $Q$
of $Z$ into parallelotopes; we call it a \emph{cubillage} for short. The
vertices of $Q$ are associated, in a natural way, with subsets of $[n]$,
forming a collection $\Spec(Q)\subseteq 2^{[n]}$, called the \emph{spectrum} of
$Q$. In the special case $d=2$, $Q$ is viewed as a rhombus tiling on a zonogon,
and it is well known due to~\cite{LZ} that the spectra of these are exactly the
maximal strongly separated set-systems in $2^{[n]}$. Similarly, the spectra of
cubillages on $Z(n,3)$ are exactly the maximal chord separated set-systems in
$2^{[n]}$, as is shown in~\cite{gal}. It turned out that this phenomenon is
extended to an arbitrary $d$: the cubillages $Q$ on $Z(n,d)$ are bijective to
the \emph{maximal by size} $(d-1)$-separated set-systems $\Sscr$ in $2^{[n]}$,
with the equality $\Spec(Q)=\Sscr$; see~\cite{GP}.

This paper is organized as follows. Section~\SEC{zon-cub} reviews definitions
and basic properties of cyclic zonotopes and cubillages. Section~\SEC{6,4}
gives a short proof of the non-purity of $\Sbold_{6,4}$, which is the crucial
case in our method of proof of Theorem~\ref{tm:GP}. As a by-product, we obtain
non-(n,4)-extendable 3-separated collections consisting of only three sets. To
show the other non-purity cases in Theorem~\ref{tm:GP}, we need to use
additional notions and constructions, which generalize those exhibited
in~\cite{DKK} for $d=3$ and are discussed in Section~\SEC{membr}. Here we
introduce a \emph{membrane} in a cubillage $Q$ on $Z(n,d)$, to be a special
$(d-1)$-dimensional subcomplex $M$ in $Q$ (when the latter is regarded as the
corresponding polyhedral complex). An important fact is that any cubillage on
$Z(n,d-1)$ can be lifted as a membrane in some cubillage on $Z(n,d)$. Also we
describe nice operations on cubillages on $Z(n,d)$ (called \emph{contraction}
and \emph{expansion} ones) that produce cubillages on $Z(n-1,d)$ and
$Z(n+1,d)$.

Section~\SEC{non-pure} utilizes this machinery to show relations between
$(n,d)$-, $(n+1,d)$-, and $(n+1,d+1)$-extendable set-systems. As a consequence,
we easily prove the remaining non-purity cases in Theorem~\ref{tm:GP}, relying
on the above result for $(n,d)=(6,4)$. Also we demonstrate in this section one
interesting class of extendable set-systems (in Proposition~\ref{pr:BD}) and
raise two open questions. Section~\SEC{invers} is devoted to additional results
involving \emph{inversions} of membranes. These objects arise as a natural
generalization of the classical notion of inversions for permutations, and
their definition for an arbitrary $(n,d)$ goes back to Manin and
Schechtman~\cite{MS}. In particular, we show that for two membranes $M$ and
$M'$ in the same cubillage on $Z(n,d)$, if any inversion of $M$ is an inversion
of $M'$, then $\Spec(M)\cup\Spec(M')$ is $(d-1)$-separated (see
Theorem~\ref{tm:MMM}).


\section{Cyclic zonotopes and cubillages}  \label{sec:zon-cub}

The objects that we deal with live in the euclidean space $\Rset^d$ of
dimension $d>1$. A \emph{cyclic configuration} of size $n\ge d$ is meant to be
an ordered set $\Xi$ of $n$ vectors $\xi_1=\xi(t_i),\ldots,\xi_n=\xi(t_n)$ in
$\Rset^d$ lying on the Veronese curve $\xi(t)=(1,t,t^2,\ldots,t^{d-1})$,
$t\in\Rset$, and satisfying $t_1<\cdots <t_n$. A useful property of $\Xi$ is
that
  \begin{numitem1} \label{eq:positive}
any $d$ vectors $\xi_{i(1)},\ldots,\xi_{i(d)}$ with $i(1)<\cdots<i(d)$ are
independent and, moreover,  $\det(A)>0$, where $A$ is the matrix whose $j$-th
column is $\xi_{i(j)}$.
  \end{numitem1}

In addition, we will also assume that $\Xi$ is $\Zset_2$-independent (i.e., all
combinations of vectors of $\Xi$ with coefficients 0,1 are different).

The configuration $\Xi$ generates the (\emph{cyclic}) \emph{zonotope}
$Z=Z(\Xi)$ in $\Rset^d$, the polytope represented as the Minkowski sum of line
segments $[0,\xi_i]$, $i=1,\ldots,n$. An object of our interest is a \emph{fine
zonotopal tiling} on $Z$, that is a subdivision $Q$ of $Z$ into $d$-dimensional
parallelotopes of which any two intersecting ones share a common face, and each
facet (a face of codimension 1) of the boundary $\bd(Z)$ of $Z$ is contained in
one of these parallelotopes. For brevity, we liberally refer to these
parallelotopes as \emph{cubes}, and to $Q$ as a \emph{cubillage}. In fact,
depending on the context, we may think of a cubillage $Q$ in two ways: either
as a set of $d$-dimensional cubes (and may write $C\in Q$ for a cube $C$ in
$Q$) or as the corresponding polyhedral complex. One can see that
  \begin{numitem1} \label{eq:cube}
each cube in $Q$ is viewed as
  $$
  \sum\nolimits_{b\in X} \xi_b + \left\{\sum(\lambda_{a(i)}\xi_{a(i)}\colon
  \; 0\le\lambda_{a(i)}\le 1,\; i=1,\ldots,d)\right\}
  $$
\noindent for some $a(1)<\cdots <a(d)$ and $X\subseteq [n]-a(1)a(2)\cdots
a(d)$.
  \end{numitem1}
Hereinafter, for a subset $\{a,\ldots,a'\}$ of $[n]$, we use the abbreviated
notation $a\cdots a'$. When $X\subset[n]$ and $a\cdots a'$ are disjoint, their
union may be denoted as $Xa\cdots a'$. Also for a set $S$ and element $i\in S$,
we may write $S-i$ for $S-\{i\}$.

For a cube $C$ in~\refeq{cube}, we say that the set $a(1)\cdots a(d)$ is the
\emph{type} of $C$, denoted as $\tau(C)$. Also, regarding the first coordinate
$x_1$ of a vector (point) $x=(x_1,\ldots,x_d)\in\Rset^d$ as its \emph{height},
we denote the lowest point $\sum_{b\in X} \xi_b$ of $C$ by $\bottom(C)$, called
the \emph{bottom} of $C$. The cells of dimensions 0 and 1 in $Q$ are called
\emph{vertices} and \emph{edges}, respectively. When needed, each edge $e$ is
directed so as to be a parallel translation of corresponding generating vector
$\xi_i$, and we say that $e$ is an edge of \emph{color} $i$, or an
$i$-\emph{edge}. This forms a directed graph on the vertices of $Q$, denoted as
$G_Q=(V_Q,E_Q)$.

The subsets $X\subseteq[n]$ are naturally identified with the corresponding
points $\sum_{b\in X} \xi_b$ in $Z$ (which are different due the
$\Zset_2$-independence of $\Xi$). This represents each vertex of $Q$ as a
subset of $[n]$, and the collection of these subsets is called the
\emph{spectrum} of $Q$ and denoted by $\Spec(Q)$.

Note that structural properties of cubillages depend on $n$ and $d$, but the
choice of a cyclic configuration $\Xi$ for these parameters is not important in
essence; so we may speak of cubillages $Q$ on a generic cyclic zonotope,
denoted as $Z(n,d)$. There are known a number of nice properties of $Q$. Among
those, two rather elementary ones are as follows:
  \begin{numitem1} \label{eq:dif_types}
all types $\tau(C)$ of cubes $C\in Q$ are different and range the set
$\binom{[n]}{d}$ of $d$-element subsets of $[n]$ (so there are exactly
$\binom{n}{d}$ cubes in $Q$); and
  \end{numitem1}
  \begin{numitem1} \label{eq:spec}
$|\Spec(Q)|=\binom{n}{0}+ \binom{n}{1}+\cdots +\binom{n}{d}$.
  \end{numitem1}

One more useful property of cubillages (which is shown for $d=3$
in~\cite[Prop.~3.5]{DKK}) and can be straightforwardly extended to an arbitrary
$d$) is:
  \begin{numitem1} \label{eq:adjacency}
suppose that for disjoint subsets $X,A$ of $[n]$, a cubillage $Q$ contains the
vertices of the form $X\cup A'$ for all $A'\subseteq A$; then these vertices
belong to a cube in $Q$.
  \end{numitem1}
In particular, if $Q$ has vertices $X$ and $Xi$, where $i\in[n]-X$, then $Q$
contains the edge connecting these vertices.

A less trivial fact is shown in~\cite{GP}; it says that
  \begin{numitem1} \label{eq:corresp}
the correspondence $Q\mapsto \Spec(Q)$ gives a bijection between the set
$\Qbold_{n,d}$ of cubillages on $Z(n,d)$ and the set $\Sbold^\ast_{n,d}$ of
maximal \emph{by size} $(d-1)$-separated collections in $2^{[n]}$ (in
particular, $|\Spec(Q)|=s_{n,d}$).
  \end{numitem1}

In light of~\refeq{LZ} and~\refeq{gal}, $\Sbold^\ast_{n,d}=\Sbold_{n,d}$ when
$d=2,3$, and in view of~\refeq{corresp}, all maximal strongly separated and
chord separated set-systems in $2^{[d]}$ are represented by the spectra of
corresponding cubillages; these facts were established for $d=2$ in~\cite{LZ}
(using equivalent terms of pseudo-line arrangements), and for $d=3$
in~\cite{gal}. On the other hand, Theorem~\ref{tm:GP} asserts that
$\Sbold^\ast_{n,d}\ne \Sbold_{n,d}$ when $d\ge 4$ and $n\ge d+2$.
 \smallskip

In our further analysis we will use additional facts on the structure of the
boundary $\bd(Z)$ of a zonotope $Z=Z(n,d)$. Let us say that a set $X\subseteq
[n]$ is a $k$-\emph{pieced cortege} if it is the union of $k$ intervals
(including the case $k=0$). As a non-difficult exercise, one can obtain the
following description for the collection $\Spec(Z)$ ($=\Spec(\bd(Z))$) of
subsets of $[n]$ represented by the vertices of $Z$:
  \begin{numitem1} \label{eq:bdZ}
for $Z=Z(n,d)$, $\Spec(Z)$ consists of exactly those sets $X\subseteq[n]$ that
are $(d-1)$-separated from any subset of $[n]$; specifically: when $d$ is even,
$\Spec(Z)$ is formed by all $d/2$-pieced corteges containing at least one of
the elements 1 and $n$ and all $k$-pieced corteges with $k<d/2$, while when $d$
is odd, $\Spec(Z)$ is formed by all $(d+1)/2$-pieced corteges containing both 1
and $n$ and all $k$-pieced corteges with $k\le (d-1)/2$.
  \end{numitem1}

In particular, $\Spec(Z)$ is included in any collection in $\Sbold_{n,d}$.

In case $n=d$, the zonotope $Z$ turns into one cube and the purity of
$\Sbold_{n,n}$ is trivial. And in case $n=d+1$, one can conclude
from~\refeq{bdZ} that there are exactly two subsets of $[n]$ that do not belong
to $\Spec(Z)$, one being formed by the odd elements, and the other by the even
elements of $[n]$, i.e., the sets $X=135\ldots$ and $Y=246\ldots$ . Clearly
they are not $(d-1)$-separated from each other. Therefore, $\Sbold_{n,n-1}$
consists of two collections $\Spec(Z_{n,n-1})\cup \{X\}$ and
$\Spec(Z_{n,n-1})\cup \{Y\}$, implying that $\Sbold_{n,n-1}$ is pure. This
together with~\refeq{LZ} and~\refeq{gal} gives ``if'' part of
Theorem~\ref{tm:GP}.

The non-purity cases of this theorem (giving ``only if'' part) are discussed in
Sections~\SEC{6,4} and~\SEC{non-pure}.

  \section{Case $(n,d)=(6,4)$} \label{sec:6,4}

This case is crucial and will be used as a base to handle the other non-purity
cases in Theorem~\ref{tm:GP} (in Sect.~\SEC{non-pure}).

Consider $Z=Z(6,4)$. By~\refeq{bdZ}, $\Spec(Z)$ consists of all intervals and
all 2-pieced corteges containing 1 or 6. A direct enumeration shows that the
number of these amounts to 52. Therefore, $2^6-52=12$ subsets of $[6]$ are not
in $\Spec(Z)$, namely:
  \begin{numitem1} \label{eq:12sets}
~24, 245, 25, 235, 35, 135, 1356, 136, 1346, 146, 1246, 246.
  \end{numitem1}
(Recall that $a\cdots b$ stands for $\{a,\ldots,b\}$.) Let $A_i$ denote $i$-th
member in this sequence (so $A_1=24$ and $A_{12}=246$). Form the collection
  $$
  \Ascr:=\Spec(Z)\cup\{A_1,A_5,A_9\}.
  $$
It consists of $52+3=55$ sets, whereas the number $s_{6,4}$ is equal to
$\binom{6}{0}+\binom{6}{1}+ \binom{6}{2}+\binom{6}{3}+\binom{6}{4}=57$. Now the
non-purity of $\Sbold_{6,4}$ is implied by the following
  \begin{lemma} \label{lm:Ascript}
$\Ascr$ is a maximal 3-separated collection in $2^{[6]}$.
  \end{lemma}
  \begin{proof}
~By~\refeq{bdZ}, any two $X\in \Spec(Z)$ and $Y\in\Ascr$ are 3-separated.
Observe that $|A_{i-1}\triangle A_i|=1$ for any $1\le i\le 12$ (where
$A_0:=A_{12}$ and $A\triangle B$ denotes the symmetric difference
$(A-B)\cup(B-A)$). Then any $A,A'\in\{A_1,A_5,A_9\}$ satisfy $|A\triangle
A'|\le 4$. This implies that $A$ and $A'$ are 3-separated. Therefore, the
collection $\Ascr$ is 3-separated.

The maximality of $\Ascr$ follows from the observation that adding to $\Ascr$
any member of $\{A_i: 1\le i\le 12,\; i\ne 1,5,9\}$ would violate the
3-separation. Indeed, a routine verification shows that $A_1$ is not
3-separated from any of $A_6,A_7,A_8$, and similarly for $A_5$ and
$\{A_{10},A_{11},A_{12}\}$, and for $A_9$ and $\{A_2,A_3,A_4\}$.
 \end{proof}

\noindent \textbf{Remark 1.} To visualize a verification in the above proof,
one can use the circular diagram illustrated in the picture below where the
sets from~\refeq{12sets} are disposed in the cyclic order. Here the sets
$A_1,A_5,A_9$ are drawn in boxes and connected by lines with those sets where
the 3-separation is violated. Note that, instead of $A_1,A_5,A_9$, one could
take in the lemma any triple of the form $A_i,A_{i+4},A_{i+8}$ (taking indices
modulo 12).
  \medskip

 \vspace{-0.0cm}
\begin{center}
\includegraphics{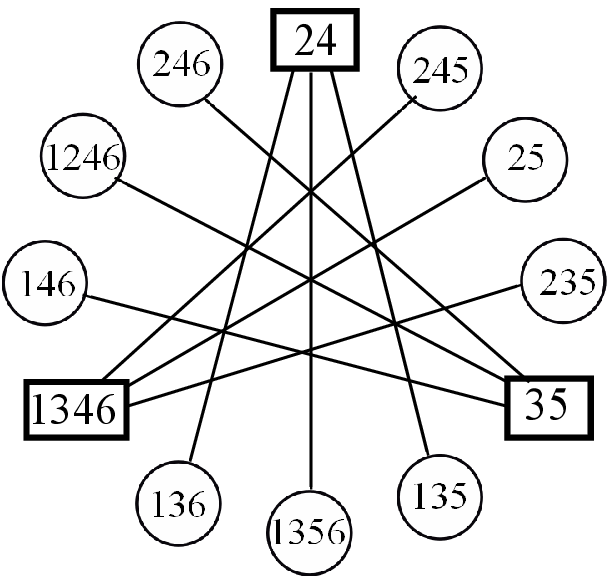}
\end{center}
\vspace{-0.0cm}

In conclusion of this section recall that a collection in $2^{[n]}$ is called
$(n,d)$-\emph{extendable} if it can be extended to a \emph{maximal by size}
$(d-1)$-separated collection in $2^{[n]}$, or, equivalently (in view
of~\refeq{corresp}), if there exists a cubillage on $Z(n,d)$ whose spectrum
includes the given collection. An immediate consequence of
Lemma~\ref{lm:Ascript} is the existence of small non-extendable 3-separated
collections.
  \begin{corollary} \label{cor:3points}
Any 3-separated triple of the form $\{A_i,A_{i+4},A_{i+8}\}$, e.g.
$\{24,35,1346\}$, is not (6,4)-extendable (defining $A_{i'}$ as above and
taking indices modulo 12).
  \end{corollary}

(Here we take into account that any maximal 3-separated collection in $2^{[6]}$
includes $\Spec(Z(6,4))$.) Note that, in view of Lemma~\ref{lm:n+1} in
Sect.~\SEC{non-pure}, this corollary implies that $\{A_i,A_{i+4},A_{i+8}\}$ is
not $(n,4)$-extendable for any $n\ge 6$.

  \section{Membranes and pies} \label{sec:membr}

For further purposes, we need additional notions, borrowing terminology
from~\cite{DKK}.


 \subsection{Membranes.}  \label{ssec:membr}
Let $\pi:\Rset^d\to\Rset^{d-1}$ be the projection along the last coordinate
vector, i.e., sending $x=(x_1,\ldots,x_d)$ to $(x_1,\ldots,x_{d-1})$. Then
$\pi(\Xi)=(\pi(\xi_1),\ldots,\pi(\xi_n))$ is again a cyclic configuration and
$Z(\pi(\Xi))$ forms a $(d-1)$-dimensional cyclic zonotope. We may use notation
$Z(n,d-1)$ for $Z(\pi(\Xi))$ and call it the projection of $Z(n,d)$.
  \smallskip

\noindent\textbf{Definition.} ~Let $Q$ be a cubillage on $Z(n,d)$. By a
\emph{membrane} in $Q$ we mean a subcomplex $M$ of $Q$ that is bijectively
projected by $\pi$ to $Z(n,d-1)$.
  \smallskip

In particular, $M$ has dimension $(d-1)$, each facet of $M$ is projected to a
$(d-1)$-dimensional ``cube'' (represented as in~\refeq{cube} where each
$\xi_\bullet$ should be replaced by $\pi(\xi_\bullet)$), and these ``cubes''
constitute a cubillage of $Z(n,d-1)$, whence the collection $\Spec(M)$ is
$(d-2)$-separated (cf.~\refeq{corresp}). A converse property says that any
cubillage can be lifted as a membrane into some cubillage of the next
dimension:
   \begin{numitem1} \label{eq:lift}
for any cubillage $Q'$ on $Z(n,d-1)$, there exists a cubillage $Q$ on $Z(n,d)$
and a membrane $M$ in $Q$ such that $\pi(M)=Q'$.
  \end{numitem1}
(This fact can be deduced from results on higher Bruhat orders and related
aspects in~\cite{MS,VK,zieg}. See also~\cite[Sec.~5]{DKK} for a direct
construction when $d=3$.)

When the choice of $Q$ as in~\refeq{lift} is not important for us, we may think
of the (abstract) membrane $M$ in $Z(n,d)$ representing a cubillage $Q'$ on
$Z(n,d-1)$, saying that $M$ \emph{is obtained by lifting} $Q'$. (To construct
such an $M$, one should take the points in $Z(n,d)$ representing the same
subsets of $[n]$ as those in $\Spec(Q')$ and then extend the corresponding
$2^{d-1}$-element subsets of these points into $(d-1)$-dimensional ``cubes'';
cf.~\refeq{adjacency}).

Two membranes are of an especial interest. For a closed set $X$ of points in
$Z=Z(n,d)$, let $X^{\rm fr}$ (resp. $X^{\rm rear}$) be the subset of points of
$X$ ``seen'' in the direction of $d$-th coordinate vector $e_d$ (resp. $-e_d$),
i.e., such that for each $x'\in\pi(X)$, this subset contains the point
$x\in\pi^{-1}(x')\cap X$ with the minimal (resp. maximal) value $x_d$. We call
it the \emph{front} (resp. \emph{rear}) side of $X$. When $X=Z$, the sides
$\Zfr$ and $\Zrear$ (respecting their cell structures) are membranes of any
cubillage on $Z$. The subcomplex $\Zfr\cap\Zrear$ is called the \emph{rim} of
$Z$ and denoted as $\Zrim$.

Borrowing terminology from~\cite{DKK}, we refer to $\pi(\Zfr)$ and
$\pi(\Zrear)$ are the \emph{standard} and \emph{anti-standard} cubillages on
$Z(n,d-1)$, denoted as $\Qst_{n,d-1}$ and $\Qant_{n,d-1}$, respectively.

As a refinement of~\refeq{bdZ}, one can characterize the spectra of $\Zrim$,
$\Zfr$ and $\Zrear$ for $Z=Z(n,d)$ as follows (a check-up is routine and we
omit it here):
   \begin{numitem1} \label{eq:Zfr-Zrear}
  \begin{itemize}
\item[(i)] when $d$ is even, $\Spec(\Zrim)$ consists of all $d/2$-pieced
corteges containing both elements 1 and $n$, and all $k$-pieced corteges with
$k<d/2$, whereas when $d$ is odd, $\Spec(\Zrim)$ consists of all
$(d-1)/2$-pieced corteges containing at least one of 1 and $n$, and all
$k$-pieced corteges with $k<(d-1)/2$;
\item[(ii)] $\Spec(\Zfr)-\Spec(\Zrim)$ consists of all $d/2$-pieced corteges containing
the element 1 but not $n$ when $d$ is even, and consists of all
$(d-1)/2$-pieced corteges containing none of 1 and $n$ when $d$ is odd;
\item[(iii)] $\Spec(\Zrear)-\Spec(\Zrim)$ consists of all $d/2$-pieced corteges containing
the element $n$ but not $1$ when $d$ is even, and consists of all
$(d+1)/2$-pieced corteges containing both 1 and $n$ when $d$ is odd.
  \end{itemize}
  \end{numitem1}


 \subsection{Pies, contraction and expansion.}  \label{ssec:pies}
For a cubillage $Q$ on $Z(n,d)$ and $i\in[n]$, let $\Pi_i=\Pi_i(Q)$ be the
subcomplex of $Q$ formed by the cubes $C$ having an edge of color $i$, called
$i$-\emph{cubes}. We refer to $\Pi_i$ as the $i$-\emph{pie} in $Q$. It has the
following nice properties (which are shown by attracting standard topological
reasonings):
   \begin{numitem1} \label{eq:pie}
   \begin{itemize}
\item[(i)] ~$\Pi_i$ is representable as the ``direct Minkowski sum''
$\{x=\beta+\alpha \colon \beta\in B_i,\; \alpha\in S_i\}$, where $B_i$ is a
subcomplex of $Q$ homeomorphic to a $(d-1)$-dimensional ball and $S_i$ is the
segment $[0,\xi_i]$;
\item[(ii)] ~removing from $Q$ the point set $\Pi-(B_i\cup B'_i)$, where
$B'_i:=B_i+\xi_i$, produces two connected components $R$ and $R'$ containing
$B_i$ and $B'_i$, respectively.
\item[(iii)] ~gluing $R$ with $R'$ shifted by $-\xi_i$, we obtain
a cubillage on the zonotope $Z(\Xi-\xi_i)$; it is denoted as $Q/i$ and called
the $i$-\emph{contraction} of $Q$;
\item[(iv)] ~$\Spec(Q/i)$ consists of all sets $X\subseteq[n]-i$ such
that at least one of $X,Xi$ is in $\Spec(Q)$.
  \end{itemize}
  \end{numitem1}

When $i=n$, the pie structure becomes more transparent. Namely, the maximality
of color $n$ in the type of each cube in $\Pi_n$ provides that the ball $B_n$
($B'_n$) is contained in the front (resp. rear) side of $\Pi_n$ (a similar fact
is also true for $i=1$ but need not hold when $1<i<n$). This implies that
  \begin{numitem1} \label{eq:Pi_n}
~$B_n$ is a membrane in the reduced cubillage ($n$-contraction) $Q/n$.
  \end{numitem1}
In other words, the $n$-contraction operation applied to $Q$, defined
by~(ii),(iii) in~\refeq{pie}, transforms the pie $\Pi_n$ into a membrane of
$Q/n$. A converse operation blows a membrane into an $n$-pie.

More precisely, let $M$ be a membrane in a cubillage $Q'$ on the zonotope
$Z'=Z({n-1},d)$. Define $Z^-(M)$ ($Z^+(M)$) to be the part of $Z'$ between
$(Z')^{\rm fr}$ and $M$ (resp. between $M$ and $(Z')^{\rm rear}$) and define
$Q^-(M)$ ($Q^+(M)$) to be the subcubillage of $Q'$ contained in $Z^-(M)$ (resp.
$Z^+(M)$). The $n$-\emph{expansion} operation for $(Q',M)$ consists in shifting
$Z^+(M)$ equipped with $Q^+(M)$ by the vector $\xi_n$ and filling the ``space
between'' $M$ and $M+\xi_n$ by the corresponding set of $n$-cubes. More
precisely, each $(d-1)$-dimensional cube $C'$ (having type $\tau(C')$ and
bottom vertex $\bottom(C'))$ in $M$ generates the $n$-dimensional cube $C=C'+
[0,\xi_n]$; so $\tau(C)=\tau(C')\cup\{n\}$ and $\bottom(C)=\bottom(C')$. Then
combining $Q^-(M)$ with the shifted subcubillage $Q^+(M)+\xi_n$ and the
``blowed membrane'' $M+[0,\xi_n]$ (forming an $n$-pie), we obtain a cubillage
on $Z(n,d)$. This cubillage is called the $n$-\emph{expansion} of $Q'$ using
$M$ and denoted as $Q_n(Q',M)$.

The $n$-contraction and $n$-expansion operations are naturally related to each
other (as a straightforward generalization to an arbitrary $d$ of
Proposition~3.4 from~\cite{DKK}):
  \begin{numitem1} \label{eq:contr-exp}
   \begin{itemize}
 \item[(i)] the correspondence $(Q',M)\mapsto Q_n(Q',M)$, where $Q'$ is a cubillage on
$Z({n-1},d)$ and $M$ is a membrane in $Q'$, gives a bijection between the set
of such pairs $(Q',M)$ in $Z(n-1,d)$ and the set $\Qbold_{n,d}$ of cubillages
on $Z(n,d)$;
  \item[(ii)]  under this correspondence, $Q'$ is the $n$-contraction $Q/n$ of
$Q=Q_n(Q',M)$ and $M$ is the image of the $n$-pie in $Q$ under the
$n$-contraction operation.
  \end{itemize}
  \end{numitem1}

  \section{Other non-purity cases} \label{sec:non-pure}

Return to proving ``only if'' part of Theorem~\ref{tm:GP}. For any $(n,d)$ with
$\min\{d-1,\,n-d+1\}\ge 3$, we have $n-6\ge d-4\ge 0$. Therefore, ``only if''
part of Theorem~\ref{tm:GP} (with $r$ replaced by $d-1$) will follow from
Lemma~\ref{lm:Ascript} concerning $(n,d)=(6,4)$ and the next two assertions on
lifting set-systems in $2^{[n]}$.
  \begin{lemma}  \label{lm:n+1}
Let $\Ascr\subseteq 2^{[n]}$. Then $\Ascr$ is $(n,d)$-extendable if and only if
$\Ascr$ is $(n+1,d)$-extendable.
  \end{lemma}

  \begin{lemma} \label{lm:n+1d+1}
Let $\Ascr\subseteq 2^{[n]}$, $n':=n+1$, and $\Ascr':=\{Xn'\colon X\in
\Ascr\}$. Then $\Ascr$ is $(n,d)$-extendable if and only if $\Ascr\cup\Ascr'$
is $(n+1,d+1)$-extendable.
  \end{lemma}

\noindent\textbf{Proof of Lemma~\ref{lm:n+1}} ~Clearly $\Ascr$ is
$(d-1)$-separated relative to $n$ if and only if it is $(d-1)$-separated
relative to $n+1$.

If $\Ascr$ is $(n,d)$-extendable, then $\Ascr\subseteq\Spec(Q)$ for some
cubillage $Q$ on $Z=Z(n,d)$. Let $Q'$ be the $(n+1)$-extension of $Q$ using as
a membrane the rear side $\Zrear$ of $Z$ (see Sect.~\SSEC{pies} for
definitions). Then $\Ascr\subseteq \Spec(Q')$, implying that $\Ascr$ is
$(n+1,d)$-extendable.

Conversely, if $\Ascr$ is $(n+1,d)$-extendable, then $\Ascr\subseteq\Spec(Q')$
for some cubillage $Q'$ on $Z(n+1,d)$. Let $Q$ be the $(n+1)$-contraction of
$Q'$ (i.e., $Q$ is a cubillage on $Z(n,d)$ obtained by shrinking the
$(n+1)$-pie in $Q'$, cf.~\refeq{pie}). Since the sets in $\Ascr$ do not contain
the element $n+1$, their corresponding vertices in $Q'$ preserve under the
$(n+1)$-contraction operation. So $\Ascr\subseteq\Spec(Q)$ and therefore
$\Ascr$ is $(n,d)$-extendable. \hfill \qed
  \medskip

\noindent\textbf{Proof of Lemma~\ref{lm:n+1d+1}} ~Let $\Ascr$ be
$(n,d)$-extendable (in particular, $\Ascr$ is $(d-1)$-separated). Take a
cubillage $Q$ on $Z(n,d)$ with $\Ascr\subseteq\Spec(Q)$. By~\refeq{lift}, there
exist a cubillage $Q'$ on $Z(n,d+1)$ and a membrane $M$ in $Q'$ such that
$Q=\pi(M)$, where $\pi$ is the projection $\Rset^{d+1}\to\Rset^d$ as in
Sect.~\SSEC{membr}. Let $Q''$ be the $n'$-expansion of $Q'$ using $M$. Then
$Q''$ is a cubillage on $Z(n+1,d+1)$. Take the $n'$-pie $\Pi_{n'}$ in $Q''$.
Then the side $B_{n'}$ of $\Pi_{n'}$ contains the vertices of $M$, whereas the
side $B'_{n'}$ contains the vertices $Xn'$ for $X\in\Spec(M)$
(cf.~\refeq{pie}(ii)). Since $\Ascr\subseteq \Spec(M)=\Spec(B_{n'})$, we have
$\Ascr'\subseteq\Spec(B'_{n'})$, whence $\Ascr\cup\Ascr'\subseteq \Spec(Q'')$.
Thus, $\Ascr\cup\Ascr'$ is $(n+1,d+1)$-extendable.

Conversely, let $\Ascr\cup\Ascr'$ be $(n+1,d+1)$-extendable (in particular, it
is $d$-separated). Then $\Ascr\cup\Ascr'\subseteq\Spec(Q)$ for some cubillage
$Q$ on $Z(n+1,d+1)$. By~\refeq{adjacency}, any $X\in\Ascr$ must be connected
with $Xn'\in\Ascr'$ by an $n'$-edge in $Q$. This implies that $\Ascr\cup\Ascr'$
is contained in the $n'$-pie of $Q$ (in which $\Ascr$ and $\Ascr'$ lie in the
sides $B_{n'}$ and $B'_{n'}$ of $\Pi_{n'}$, respectively). Then the
$n'$-contraction operation applied to $Q$ transforms $Q$ and $\Pi_{n'}$ into a
cubillage $Q'$ on $Z(n,d+1)$ and a membrane $M$ in $Q'$, preserving the
vertices of $B_{n'}$. Hence $\Ascr\subseteq\Spec(M)$, and now taking the
cubillage $Q'':=\pi(M)$ on $Z(n,d)$, we obtain $\Ascr\subseteq\Spec(Q'')$, as
required. \hfill \qed
  \medskip

Note that Lemmas~\ref{lm:n+1} and~\ref{lm:n+1d+1} can also be used to
demonstrate an interesting class of extendable set-systems, as follows.
  \begin{prop} \label{pr:BD}
Let $k\in[n]$, $\Bscr\subseteq 2^{[k]}$, $D\subseteq \{k+1,\ldots,n\}$, and
$d:=|D|$. Let $\Cscr$ consist of the sets of the form $B\cup D'$ for all
$B\in\Bscr$ and $D'\subseteq D$. Suppose that $\Bscr$ is 2-separated. Then
$\Cscr$ is $(n,d+3)$-extendable.
  \end{prop}
  \begin{proof}
We use induction on $n+d$. When $d=0$, ~$\Cscr$ becomes $\Bscr$, and therefore
it is $(n,3)$-extendable (by~\refeq{gal}).

So assume that $d\ge 1$. Let $p$ be the maximal element in $D$. If $p<n$ then
$\Cscr$ is contained in $2^{[n-1]}$. By Lemma~\ref{lm:n+1}, $\Cscr$ is
$(n,d+3)$-extendable if and only if it is $(n-1,d+3)$-extendable, and we can
apply induction.

Now let $p=n$. Form $\Cscr':=\{X\in\Cscr\colon n\notin X\}$ and
$\Cscr'':=\{X\in\Cscr\colon n\in X\}$. Then $\Cscr'\cap\Cscr''=\emptyset$,
~$\Cscr'\cup\Cscr''=\Cscr$, and the construction of $\Cscr$ implies that
$\Cscr''=\{Xn\colon X\in\Cscr'\}$. Therefore, by Lemma~\ref{lm:n+1d+1}, $\Cscr$
is $(n,d+3)$-extendable if and only if $\Cscr'$ is $(n-1,d+2)$-extendable.
Since $\Cscr'$ consists of the sets of the form $B\cup D'$ for all $B\in\Bscr$
and $D'\subseteq D\cap [n-1]$, we can apply induction.
  \end{proof}

As a special case in this proposition, we obtain the following
 \begin{corollary} \label{cor:extend_cube}
For any $D\subseteq [n]$ with $|D|\le d$, the collection $\{X\subseteq D\}$
(forming the vertex set of a ``cube'') is $(n,d)$-extendable.
  \end{corollary}
(In this case one should take as $\Bscr$ the ``cube'' on three smallest
elements of $D$.)

Thus, one cube of dimension $\le d$ within a zonotope $Z(n,d)$ can always be
extended to a cubillage. In contrast, as we have seen earlier (cf.
Corollary~\ref{cor:3points}), a triple of (duly separated) ``cubes'' of
dimension 0 need not be extendable. In light of these facts, we can address the
following open question:
  \begin{description}
\item[$(O1)$]: Whether or not any two ``cubes'' $\Cscr=\{A\cup X\colon X\subseteq D\}$ and
$\Cscr'=\{A'\cup Y\colon Y\subseteq D'\}$, where $A,A',D,D'\subset[n]$,
~$|D|,|D'|\le d$, and $\Cscr\cup\Cscr'$ is $(d-1)$-separated, can be extended
to a cubillage on $Z(n,d)$?
  \end{description}
A similar open question concerns a membrane and a cube:
   \begin{description}
\item[$(O2)$]: Whether or not any pair consisting of a membrane $M$ in $Z(n,d)$
and a ``cube'' $\Cscr=\{A\cup X\colon X\subseteq D\}$, where $A,D\subset[n]$,
~$|D|\le d$, and $\Spec(M)\cup\Cscr$ is $(d-1)$-separated, can be extended to a
cubillage on $Z(n,d)$?
  \end{description}

  \section{Inversions} \label{sec:invers}

Inversions discussed in this section arise as a natural generalization of the
classical notion of inversions in elements (permutations) of a symmetric group
$\Sscr_n$, inspired by the observation that a permutation of $[n]$ can be
interpreted as a membrane in the zonogon $Z(n,2)$. We will use two ways to
define inversions, which are shown to be equivalent. The first one is of a
geometric flavor, as follows.
  \medskip

\noindent\textbf{Definitions.} ~Consider a cubillage $Q$ on $Z=Z(n,d)$, a
membrane $M$ in $Q$, and the corresponding cubillage $Q':=\pi(M)$ on
$Z(n,d-1)$. A $d$-tuple $K\in\binom{[n]}{d}$ is called \emph{inversive}, or an
\emph{inversion}, for $M$, as well as for $Q'$, if the cube $C$ of $Q$ having
type $K$ lies \emph{before} $M$, i.e., in the region $Z^-(M)$ between $\Zfr$
and $M$ (see Sect.~\SSEC{membr} for definitions). Otherwise (when the cube
$C\in Q$ with $\tau(C)=K$ lies in the region $Z^+(M)$ between $M$ and
$\Zrear$), we say that $K$ is \emph{straight} for $M$ (and for $Q'$).
  \smallskip

An important fact is that the set of inversions for $M$ does not depend on the
choice of a cuballage $Q$ on $Z$ that contains $M$ as a membrane (see Remark~2
below); we denote this set as $\Inver(M)$, or $\Inver(Q')$.

In particular, the smallest case $\Inver(M)=\emptyset$ (the largest case
$\Inver(M)=\binom{[n]}{d}$) happens when $M=\Zfr$ (resp. $M=\Zrear$), or, in
terms of cubillages, when $Q'=\pi(M)$ becomes the standard cubillage
$\Qst_{n,d-1}$ (resp. the anti-standard cubillage $\Qant_{n,d-1}$). (Note
that~\cite{zieg} exhibits necessary and sufficient conditions on a collection
in $\binom{[n]}{d}$ to be the set of inversions for a membrane, but we do not
use this in what follows.)

The second way to define inversions relies on a natural binary relations on
cubes of a cubillage. More precisely, given a cubillage $\tilde Q$ on $Z(\tilde
n,\tilde d)$, we say that a cube $C\in\tilde Q$ \emph{immediately precedes} a
cube $C'\in\tilde Q$ if their sides $C^{\rm rear}$ and $(C')^{\rm fr}$ share a
facet (a $(\tilde d-1)$-dimensional face). Then for $C,C'\in\tilde Q$, we write
$C\prec_{\tilde Q} C'$ and say that $C$ \emph{precedes} $C'$ if there is a
sequence $C=C_0,C_1,\ldots,C_k=C'$ such that $C_{i-1}$ immediately precedes
$C_i$ for each $i$. It can be shown rather easily that the relation
$\prec_{\tilde Q}$ is a partial order (arguing in spirit of the proof of
Lemma~4.2 in~\cite{DKK} for $d=3$).

A behavior of this order under contraction operations on pies (defined in
Sect.~\SSEC{pies}) is featured as follows.
  \begin{lemma} \label{lm:prec-contr}
Let $Q^\ast$ be the $i$-contraction of a cubillage $\tilde Q$ on $Z(\tilde
n,\tilde d)$, where $i\in[\tilde n]$, and suppose that $D,D'$ are cubes in
$Q^\ast$ such that $D\prec_{Q^\ast} D'$. Then the cubes $C,C'\in\tilde Q$ with
$\tau(C)=\tau(D)$ and $\tau(C')=\tau(D')$ satisfy $C\prec_{\tilde Q} C'$.
  \end{lemma}
  \begin{proof}
Let $B_i,R,R'$ be as in~\refeq{pie}(i),(ii) for the $i$-pie $\Pi_i$ in $\tilde
Q$. It suffices to consider the case when $D$ immediately precedes $D'$. Then
four cases are possible: (a) both $D,D'$ lie in the part $R$ of $Q^\ast$, (b)
both $D,D'$ lie in the part $R'-\xi_i$ of $Q^\ast$ corresponding to $R'$ in
$\tilde Q$, (c) $D$ lies in $R$, and $D'$ in $R'-\xi_i$, or (d) $D$ lies in
$R'-\xi_i$, and $D'$ in $R$.

In case~(a), we have $C=D$ and $C'=D'$, while in case~(b), $C=D+\xi_i$ and
$C'=D'+\xi_i$. So in both cases $C$ immediately precedes $C'$. In case~(c), we
have $C=D$ and $C'=D'+\xi_i$. Also $D^{\rm rear}$ and $(D')^{\rm fr}$ share a
facet $F$ contained in $B_i$. Therefore, the pie $\Pi_i$ has the $i$-cube $C''$
that is the sum of $F$ and the segment $[0,\xi_i]$. One can see that $C^{\rm
rear}\cap (C'')^{\rm fr}=F$ and $(C'')^{\rm rear}\cap (C')^{\rm fr}=F+\xi_i$.
Then $C$ immediately precedes $C''$, and $C''$ immediately precedes $C'$. This
implies $C\prec_{\tilde Q} C'$. Finally, in case~(d), $C=D+\xi_i$ and $C'=D'$.
Also $D^{\rm rear}$ and $(D')^{\rm fr}$ share a facet $F$ in $B_i$. Again,
$\Pi_i$ has the $i$-cube $C''$ that is the sum of $F$ and the segment
$[0,\xi_i]$. But now we have $C^{\rm rear}\cap (C'')^{\rm fr}=F+\xi_i$ and
$(C'')^{\rm rear}\cap (C')^{\rm fr}=F$. Then $C\prec_{\tilde Q}C''\prec_{\tilde
Q} C'$, implying $C\prec_{\tilde Q} C'$, as required.
  \end{proof}

Now return to $Q,M,Q'$ as above. Fix a cube $C\in Q$ and let $K:=\tau(C)$.
Suppose we apply to $Q$ the $i$-contraction operation with $i\in[n]-K$. One can
see that under this operation $M$ turns into a membrane $M'$ in the resulting
cubillage $Q/i$, and comparing the location of $C$ relative to $M$ with that of
the ``image'' of $C$ relative to $M'$, one can see that the status of $K$
preserves, i.e., $K$ is inversive for $M'$ if and only if so is for $M$.

By applying, step by step, the $i$-contraction operations to all $i\in[n]-K$,
we produce from $Q$ the cubillage $\hat Q:=Q/([n]-K)$ consisting of a single
cube $\hat C$ having type $K$. Accordingly, $M$ and $Q'$ turn into the membrane
$\hat M:=M/([n]-K)$ in $\hat Q$ and its projection $\hat Q':=\pi(\hat M)$,
respectively. Since $\hat Q$ has exactly two membranes, namely, ${\hat C}^{\rm
fr}$ and ${\hat C}^{\rm rear}$, and since the status of $K$ preserves during
the contraction process, we can conclude that
  \begin{numitem1} \label{eq:status}
if $K$ is inversive (straight) for $M$, then the reduced membrane $\hat
M:=M/([n]-K)$ is isomorphic to the rear (resp. front) side of a cube of type
$K$.
  \end{numitem1}

Assuming that $K$ consists of elements $k_1<\cdots< k_d$, we will write
$Z(K,d-1)$ for the zonotope generated by $\xi'_p:=\pi(\xi_{k_p})$,
$p=1,\ldots,d$ (where, as before, $\xi_\bullet$ is a generator from $\Xi$). In
view of $|K|=d$, there exist exactly two cubillages on $Z':=Z(K,d-1)$, namely,
the standard and anti-standard cubillages, denoted as $\Qst_{K,d-1}$ and
$\Qant_{K,d-1}$, respectively (which correspond to the standard and
anti-standard cubillages in $Z(d,d-1)$).

Since $\Qst_{K,d-1}$ and $\Qant_{K,d-1}$ are the projections of $\Cfr$ and
$\Crear$, respectively, where $C$ is a cube of type $K$ (viz. $C=Z(K,d)$),
\refeq{status} implies that
  \begin{numitem1} \label{eq:inv-antist}
$K\in\binom{[n]}{d}$ is an inversion for a membrane $M$ in $Z(n,d)$ if and only
if $\pi(M/([n]-K))$ is the anti-standard cubillage $\Qant_{K,d-1}$ on
$Z(K,d-1)$.
  \end{numitem1}

This and Lemma~\ref{lm:prec-contr} lead to a description of inversions for $M$
in terms of the partial order $\prec_M$, giving the second (``intrinsic'') way
to characterize $\Inver(M)$. Following~\cite{MS}, for $K\in\binom{[n]}{d}$,
define $\Pack(K)$ to be the set $\{K-i\colon i\in K\}$ of $(d-1)$-element
subsets of $K$, called the \emph{packet} of $K$. Then each $K'\in\Pack(K)$ is
the type of some cube in $M$. For convenience, we use the same notation
$\prec_M$ for the corresponding types; so if $C,C'\in M$ and $C\prec_M C'$, we
may write $\tau(C)\prec_M\tau(C')$.
  \begin{prop} \label{pr:lex-antilex}
Let $K\in\binom{[n]}{d}$ consist of elements $k_1<\cdots<k_d$ and let $M$ be a
membrane in $Z(n,d)$. Then the elements of $\Pack(K)$ occur in the
lexicographic order
  $$
  (K-k_d)\prec_M (K-k_{d-1})\prec_M\cdots\prec_M (K-k_1)
    $$
if $K$ is straight for $M$, and in the anti-lexicographic order
  $$
  (K-k_1)\prec_M (K-k_2)\prec_M\cdots\prec_M (K-k_d)
    $$
if $K$ is inversive for $M$.
  \end{prop}
  \begin{proof}
In view of Lemma~\ref{lm:prec-contr}, the required relations for $\prec_M$
would follow from similar relations for $\prec_{M'}$, where $M':=M/([n]-K)$.
Moreover, since the relations $\prec_{M'}$ and $\prec_{\pi(M')}$ on $\Pack(K)$
are the same, it suffices to consider cubillages on $Z(K,d-1)$, or,
equivalently, on $Z(d,d-1)$. In other words, we have to show that for the sets
$A_i:=[d]-i$, $i=1,\ldots,d$:
  \begin{eqnarray}
  &&A_d\prec' A_{d-1}\prec'\cdots\prec' A_1, \quad\mbox{and} \label{eq:lexAd}
                            \\
  &&A_1\prec'' A_2\prec''\cdots\prec'' A_d,  \label{eq:antilexAd}
  \end{eqnarray}
where $\prec'$ ($\prec''$) denotes the order in $Q':=\Qst_{d,d-1}$ (resp.
$Q'':=\Qant_{d,d-1}$).

To show this, we first specify the spectra of $Q'$ and $Q''$. Let $X$ ($Y$) be
the set of elements $i\in[d]$ with $d-i$ odd (resp. even). Note that $X$ and
$Y$ are not $(d-2)$-separated from each other; so one of $X,Y$ belongs to
$\Spec(Q')$, and the other to $\Spec(Q'')$ (taking into account that each of
$\Spec(Q')$ and $\Spec(Q'')$ is $(d-2)$-separated and that $|\Spec(Q')|
=|\Spec(Q'')|=\binom{d}{0}+\binom{d}{1}+\cdots \binom{d}{d-1}=2^d-1$;
cf.~\refeq{corresp} and~\refeq{spec}). Using $\Spec(Q')=\Spec(\Zfr)$ and
$\Spec(Q'')=\Spec(\Zrear)$ for the ``cube'' $Z=Z(d,d)$ and
considering~\refeq{Zfr-Zrear}(ii),(iii), one can conclude that
   $$
 X\in\Spec(Q') \quad\mbox{and} \quad Y\in\Spec(Q'').
   $$

Let us prove~\refeq{lexAd}. The cubillage $Q'$ is formed by $d$ cubes
$C_1,\ldots,C_d$ of types $A_1,\ldots,A_d$, respectively, each of which must
contain the unique vertex of $Q'$ lying in the interior of $Z'=Z(d,d-1)$,
namely, the vertex $X$ (viz. $\sum(\pi(\xi_i)\colon i\in X)$). It follows that
in the digraph $G_{Q'}$ (defined in Sect.~\SEC{zon-cub}),
   \begin{numitem1} \label{eq:GQX}
the vertex $X$ is incident to $d$ edges $a_1,\ldots,a_d$ of $G_{Q'}$, where
each $a_i$ is an $i$-edge, and $a_i$ \emph{enters} (resp. \emph{leaves}) $X$ if
$i\in X$ (resp. $i\in[d]-X$);
  \end{numitem1}
  \begin{numitem1} \label{eq:Ci}
for $i=1,\ldots,d$, the cube $C_i$ contains all edges in
$E:=\{a_1,\ldots,a_d\}$ except for $a_i$.
   \end{numitem1}

Consider ``consecutive'' cubes $C_i,C_{i+1}$ ($1\le i<d$). They share a facet,
namely, the one lying in the hyperplane $H_i$ spanned by the edge set
$E-\{a_i,a_{i+1}\}$. The required relation $A_{i+1}\prec' A_i$ in~\refeq{lexAd}
can be reformulated as:
   \begin{description}
\item[($\ast$)] when seeing in the direction of the last coordinate vector,
the cube $C_i$ is located \emph{behind} $H_i$ (whereas $C_{i+1}$ is located
\emph{before} $H_i$).
  \end{description}

  To see~$(\ast)$, for $j=1,\ldots, d$, denote $\pi(\xi_j)$ by $\varphi_j$, and
define the vector $\bar\varphi_j$ to be $-\varphi_j$ if $j\in X$, and
$\varphi_j$ otherwise. We write $D(\beta,\ldots,\beta')$ for the determinant of
the matrix formed by a sequence $\beta,\ldots,\beta'$ of $d-1$ column vectors
in $\Rset^{d-1}$. Note that $(\ast)$ says that the edge $a_{i+1}$ of $C_i$ is
located behind $H_i$ (and the edge $a_i$ of $C_{i+1}$ before $H_i$). One can
realize that this location corresponds to the relation
   \begin{description}
\item[($\ast\ast$)] ~$D:=D(\varphi_1,\varphi_2,\ldots,\varphi_{i-1},\varphi_{i+2},
\varphi_{i+3},\ldots,\varphi_d,\bar\varphi_{i+1})>0$.
   \end{description}

Now validity of~$(\ast\ast)$ follows from
   \begin{multline*}
D=(-1)^{d-i-1} D(\varphi_1,\ldots,\varphi_{i-1},\bar\varphi_{i+1},
\varphi_{i+2},\ldots,\varphi_d)   \\
   =D(\varphi_1,\ldots,\varphi_{i-1},\varphi_{i+1},\varphi_{i+2},\ldots,\varphi_d)>0
   \end{multline*}
(taking into account~\refeq{positive} and the fact that
$\bar\varphi_{i+1}=-\varphi_{i+1}$ if and only if $d-i-1$ is odd).

To show~\refeq{antilexAd}, we argue in a similar way, replacing~\refeq{GQX} by:
  \begin{numitem1}
in the graph $G_{Q''}$, the vertex $Y$ is incident to $d$ edges
$b_1,\ldots,b_d$, where each $b_i$ is an $i$-edge, and $b_i$ enters (resp.
leaves) $Y$ if $i\in Y$ (resp. $i\in[d]-Y$).
  \end{numitem1}

Using this and the fact that $Y$ if formed by elements $i\in[d]$ with $d-i$
even, one shows that for each $i$, the cube $C_i$ of type $A_i$ is located
\emph{before} the hyperplane separated $C_i$ and $C_{i+1}$ (cf.~$(\ast)$),
and~\refeq{antilexAd} follows.
  \end{proof}

\noindent \textbf{Remark 2.} The above proposition implies that the
``geometric'' definition of $\Inver(M)$ (given in the beginning of this
section) does not depend on the choice of a cubillage containing $M$ as a
membrane. Next, for a membrane $M$ in $Z(n,d)$, we alternatively could give a
``packet'' definition for straight and inversive tuples $\binom{[n]}{d}$ in a
spirit of the statement in this proposition, and then come to the ``geometric''
characterization by reversing reasonings in the above proof. This alternative
way to define $\Inver(M)$ matches the classical definition due to Manin and
Schechtman (cf. Theorem~3 in~\cite{MS}). Recall that they introduced a ``packet
admissible'' total order $\prec$ on $\binom{[n]}{d-1}$, which means that for
each tuple $K\in\binom{[n]}{d}$, the elements of ${\rm Pac}(K)$ become ordered
by $\prec$ either lexicographically or anti-lexicographically, and in the
latter case, $K$ is said to be an inversion for $(\binom{[n]}{d-1},\prec)$.
(Compare $\prec$ with $\prec_M$.) \smallskip

Note also that the method of proof of Proposition~\ref{pr:lex-antilex} enables
us to reveal one more useful fact.
  \begin{prop} \label{pr:inv-str}
Let $M$ be a membrane in $Z(n,d)$ and let $K\in\binom{[n]}{d}$ consist of
elements $k_1<\cdots <k_d$. Then:
   \begin{itemize}
\item[{\rm(i)}] $K$ is inversive for $M$ if and only if there is $X\in \Spec(M)$
such that $X\cap K=\{k_i\colon d-i$ odd$\}=:K^{\rm odd}$;
\item[{\rm(ii)}] $K$ is straight for $M$ if and only if there is $Y\in \Spec(M)$
such that $Y\cap K=\{k_i\colon d-i$ even$\}=:K^{\rm even}$.
   \end{itemize}
   \end{prop}
  \begin{proof}
~Let $[n]-K=\{j_1,j_2,\ldots,j_{n-d}\}$ and form the sequence
$M_0=M,M_1,\ldots,M_{n-d}$ of membranes in the corresponding zonotopes, where
$M_i=(M_{i-1})/j_i$. So $M':=M_{n-d}$ is a membrane in the final zonotope
$Z':=Z(K,d)$ (a single cube). We know that if $K$ is inversive for $M$, then
$M'=(Z')^{\rm rear}$ and $M'$ contains the vertex $K^{\rm odd}$, whereas if $K$
is straight for $M$, then $M'=(Z')^{\rm fr}$ and $M'$ contains the vertex
$K^{\rm even}$ (cf. the proof of Proposition~\ref{pr:lex-antilex}). Now the
result follows by observing that for $1\le i\le n-d$, if the membrane $M_i$ has
a vertex $A$, then the previous membrane $M_{i-1}$ has a vertex $A'$ of the
form $A$ or $A\cup\{j_i\}$.
  \end{proof}

As a consequence, we obtain the following result.
  \begin{theorem} \label{tm:MMM}
Let $M_1,\ldots,M_p$ be membranes in $Z(n,d)$ such that $\Inver(M_1)\subset
\cdots\subset \Inver(M_p)$. Then the collection $\Spec(M_1)\cup \ldots \cup
\Spec(M_p)$ is $(d-1)$-separated.
  \end{theorem}
    \begin{proof}
Suppose that this is not so. Then for some $i<j$, there exist $X\in\Spec(M_i)$
and $Y\in\Spec(M_j)$ that are not $(d-1)$-separated from each other. Therefore,
there exist elements $i_1<i_2<\cdots <i_{d+1}$ of $[n]$ that alternate in $X-Y$
and $Y-X$. Let for definiteness the elements $i_k$ with $k$ odd are contained
in $X-Y$ (and the other in $Y-X$). Consider the $d$-element sets
$K:=\{i_1,\ldots,i_d\}$ and $K':=\{i_2,\ldots,i_{d+1}\}$. Then, by
Proposition~\ref{pr:inv-str}, $K$ is straight for one, and inversive for the
other membrane among $M_i,M_j$. But the behavior of $M_i,M_j$ relative to $K'$
is opposite. Thus, neither $\Inver(M_i)\subset\Inver(M_j)$ nor
$\Inver(M_j)\subset\Inver(M_i)$ is possible; a contradiction.
    \end{proof}

We finish this section with two applications.
\smallskip

1) Let $M,N$ be two membranes with $\Inver(M)\subset \Inver(N)$ in $Z=Z(n,d)$.
By Theorem~\ref{tm:MMM}, the collection $\Cscr:=\Spec(M)\cup\Spec(N)$ is
$(d-1)$-separated; so it is tempting to hope that $\Cscr$ is extendable to a
maximal by size $(d-1)$-separated set-system, or, equivalently, that there
exists a cubillage $Q$ on $Z$ containing both membranes. We can try to
construct such a $Q$ by filling the region $Z^-(M)$ (between $\Zfr$ and $M$)
with a ``partial'' cubillage $Q'$, and filling the region $Z^+(N)$ (between $N$
and $\Zrear$) with a ``partial'' cubillage $Q''$ (such $Q',Q''$ exist
by~\refeq{lift}). But what is about the rest of $Z$ between $M$ and $N$,
denoted as $Z(M,N)$? (Note that $\Inver(M)\subset\Inver(N)$ provides that $M$
lies within $Z^-(N)$.)

Let us say that $M,N$ are \emph{agreeable} if the collection
$\Spec(M)\cup\Spec(N)$ is $(n,d)$-extendable, i.e., a cubillage on $Z$
containing both $M,N$ (equivalently, a ``partial'' cubillage filling $Z(M,N)$)
does exist. Ziegler~\cite{zieg} explicitly constructed two membranes $M,N$ in
the zonotope $Z(8,4)$ such that $\Inver(M)\subset\Inver(N)$ but $M,N$ are not
agreeable (in our terms) . This together with Theorem~\ref{tm:MMM} implies that
the set system $\Sbold_{8,4}$ is not pure (the latter fact was omitted
in~\cite{zieg}). (Compare $(O2)$ in the end of the previous section that
considers the union of a membrane and a cube.)
  \smallskip

2) In light of the above result for $d=4$, Ziegler asked about the existence of
two non-agreeable membranes in dimension 3. Answering this question, Felsner
and Weil~\cite{FW} proved that for an arbitrary $n$, any two membranes $M,N$
with $\Inver(M)\subset\Inver(N)$ in $Z(n,3)$ are agreeable. Note that the proof
in~\cite{FW} attracted a non-trivial combinatorial techniqies. An alternative
proof immediately follows from Galashin's result in~\cite{gal} (mentioned
in~\refeq{gal}) and Theorem~\ref{tm:MMM}.


\end{document}